\documentclass[
11pt,twoside, a4paper,
]{amsart}

\title[Homogeneous spaces of Hilbert type]
{  
{ Homogeneous spaces of Hilbert type}
}
\author{Mikhail Borovoi}
\address{Raymond and Beverly Sackler School of Mathematical Sciences,
Tel Aviv University, 6997801 Tel Aviv, Israel}
\email{borovoi@post.tau.ac.il}
\thanks{Partially supported by the Hermann Minkowski Center for Geometry}


\usepackage[arrow,matrix]{xy}

\usepackage{mathptmx}
\usepackage{amscd}
\usepackage{amssymb}
\usepackage{eucal}
\usepackage{amsxtra}

\usepackage{verbatim}
\usepackage{url}

\usepackage{enumerate}

\DeclareSymbolFont{rsfs}{U}{rsfs}{m}{n}
\DeclareSymbolFontAlphabet{\mathcal}{rsfs}

\DeclareFontEncoding{OT2}{}{} 
\DeclareTextFontCommand{\textcyr}{\fontencoding{OT2}
    \fontfamily{wncyr}\fontseries{m}\fontshape{n}\selectfont}
\newcommand{\Sha}{{\textcyr{Sh}}}
\def\Sh{\Sha}
\def\Be{\textcyr{B}}

\usepackage{verbatim}

\theoremstyle{plain}

\newtheorem{theorem}{Theorem}[section]
\newtheorem{proposition}[theorem]{Proposition}
\newtheorem{lemma}[theorem]{Lemma}

\newtheorem{conditional-result}[theorem]{Conditional Result}

\newtheorem{theorem?}{Theorem(?)}[section]
\newtheorem{proposition?}[theorem]{Proposition(?)}
\newtheorem{lemma?}[theorem]{Lemma(?)}
\newtheorem{corollary?}[theorem]{Corollary(?)}

\newtheorem*{theorem*}{Theorem}
\newtheorem*{proposition*}{Proposition}
\newtheorem*{lemma*}{Lemma}
\newtheorem*{corollary*}{Corollary}
\newtheorem*{question*}{Question}
\newtheorem*{conjecture*}{Conjecture}
\newtheorem*{claim*}{Claim}

\theoremstyle{definition}

\newtheorem{notation}[theorem]{Notation}
\newtheorem*{definition*}{Definition}
\newtheorem*{example*}{Example}

\theoremstyle{remark}

\newtheorem*{remark*}{Remark}
\newtheorem*{Remarks*}{Remarks}
\newenvironment{remarks*}{\begin{Remarks*}\nopagebreak[4]
\rule{1em}{0ex}\par 
\begin{theoremlist}}%
{\end{theoremlist}\end{Remarks*}}

\newcommand{\isoto}{\overset{\sim}{\to}}

\newcommand{\into}{\hookrightarrow}
\newcommand{\onto}{\twoheadrightarrow}

\DeclareMathOperator{\coker}{coker}

\DeclareMathOperator{\Spec}{Spec}

\DeclareMathOperator{\Hom}{Hom}

\newcommand{\uu}{\mathrm{u}}

\def\red{\mathrm{red}}
\def\tor{{\mathrm{tor}}}

\def\sss{{\mathrm{ss}}}

\def\mult{{\mathrm{mult}}}
\def\ssu{{\mathrm{ssu}}}

\DeclareMathOperator{\Gal}{Gal}

\DeclareMathOperator{\Br}{Br}
\newcommand{\Bra}{\Br_\mathrm{a}}

\newcommand{\kbar}{{\bar{k}}}

\def\Br{{\rm Br\ }}

\def\kbar{{\overline{k}}}
\def\Gbar{{\overline{G}}}
\def\Hbar{{\overline{H}}}
\def\Xbar{{\overline{X}}}
\def\xbar{{\overline{x}}}
\def\Om{{\mathcal{V}_\infty(k)}}
\def\kinf{{k_\infty}}

\def\Z{{\mathbb{Z}}}
\def\Q{{\mathbb{Q}}}
\def\A{{\mathbb{A}}}

\def\sG{{\mathcal{G}}}

\def\ms{\medskip\par}

\def\Br{{\rm Br}}
\def\Bro{{\Br_1}}
\def\H{{\mathbb{H}}}
\def\sV{{\mathcal{V}}}

\def\Aut{{\rm Aut\ }}

\def\Ghat{{\widehat{G}}}
\def\Hhat{{\widehat{H}}}

\def\GG{{\mathbb{G}}}

\begin{document}

\maketitle

\begin{abstract}
Let $k$ be a global field.
Let $G$ be a connected linear algebraic $k$-group,
assumed reductive when $k$ is a function field.
It follows from a result of a paper  by Bary-Soroker, Fehm and Petersen
that when $H$ is a smooth connected $k$-subgroup of $G$, the quotient space $G/H$ is of Hilbert type.
We prove a similar result for certain non-connected $k$-subgroups $H$ of $G$.
In particular, we prove that if $G$ is a simply connected $k$-group over a number field $k$,
and $H$ is an abelian $k$-subgroup of $G$, not necessarily connected,
then $G/H$ is of Hilbert type.

\end{abstract}

\section{Introduction}\label{s:intro}

Let $k$ be a field.
By a $k$-variety we mean a  geometrically integral, separated scheme of finite type over $k$.
Let $X$ be a $k$-variety.
A subset $\Theta\subset X(k)$ is said to be of type $(C_1)$ if there is
a closed $k$-subvariety $Y\subset X$, \ $Y\neq X$, with $\Theta\subset Y(k)$.
A subset $\Theta\subset X(k)$ is said to be of type $(C_2)$
if there is a $k$-variety $X'$ with $\dim X'=\dim X$
and a dominant separable morphism $\pi\colon X'\to X$ of degree $\ge 2$
with $\Theta\subset \pi(X'(k))$.
A subset $\Theta$ of $X(k)$ is called {\em thin} if
it is contained in a finite union of subsets of types $(C_1)$ or $(C_2)$,
cf. Serre \cite[Def.~3.1.1]{Serre} and Bary-Soroker, Fehm and Petersen \cite[Section 2]{BFP}.
It follows from the definition that if we have  a finite family of thin subsets $\Theta_1,\dots,\Theta_r\subset X(k)$,
then their union $\Theta_1\cup\dots\cup \Theta_r$ is again a thin subset.

The notion of a variety {\em of Hilbert type} was introduced by Colliot-Th\'el\`ene and Sansuc \cite{CS}.
By definition, see Serre \cite[Def.~3.1.2]{Serre}, a $k$-variety $X$ is of Hilbert type if $X(k)$ is not thin.
A field $k$ is called Hilbertian if the affine  line $\A^1_k$ is of Hilbert type.
All the global fields (number fields and function fields) are Hilbertian,
see  \cite[Thm.~3.4.1]{Serre} and  Fried and Jarden \cite[Thm.~13.4.2]{FJ}.

It was proved  by Colliot-Th\'el\`ene and Sansuc in \cite[Cor.~7.15]{CS}
that  any connected reductive algebraic $k$-group over a Hilbertian field $k$ is of Hilbert type.
Moreover, we have the following more general result, cf. \cite{BFP}:

 \begin{proposition}\label{p:quot-connected}
 Let $k$ be a Hilbertian field.
 Let $G$ be a smooth connected linear algebraic $k$-group,
 assumed reductive when $k$ is not perfect.
 Let $H\subset G$ be a smooth connected $k$-subgroup.
 Then the quotient space $G/H$ is a $k$-variety of Hilbert type.
 \end{proposition}

 \begin{proof}
 When $k$ is perfect, see \cite[Cor.~4.5]{BFP}.
 When $k$ is not perfect and $G$ is reductive, the proof is similar.
 \end{proof}

 In this note we consider the case when $k$ is a {\em global field} and $G$ is a smooth connected linear algebraic $k$-group.
 We prove that  $G/H$ is of Hilbert type for certain   \emph{non-connected and non-smooth}  subgroups $H\subset G$.
 In particular, we prove the following theorem:

 \begin{theorem}\label{t:abelian}
Let $G$ be a simply connected $k$-group over a number field $k$,
and let $H\subset G$ be an abelian $k$-subgroup, not necessarily connected.
Then $G/H$ is of Hilbert type.
\end{theorem}

 \begin{proof}
 This theorem is a special case of Theorem \ref{t:quotient} below.
 \end{proof}

 The plan of the rest of the note is as follows.
 In the end of the Introduction we give some notation.
  In Section \ref{s:number} we consider homogeneous spaces over a number field.
 The main result of this section, Theorem \ref{t:quotient}, implies Theorem \ref{t:abelian}, and its
 proof uses \cite[Thm.~1.1]{BFP}. In Section \ref{s:function} we consider homogeneous spaces over a global function field.
 The main result of this section is Theorem \ref{t:quotient-pos},
 the proof of which is similar to and  easier than that of Theorem \ref{t:quotient},
 and uses \cite[Cor.~1.2]{BFP}.
 In Section \ref{s:weak} we give an alternative proof of Theorem \ref{t:quotient}.
 We show that a homogeneous space $G/H$ as in Theorem  \ref{t:quotient}
 has the WWA (weak weak approximation) property, cf. \cite[Def.~3.5.6]{Serre},
 and therefore, by a theorem of Ekedahl and Colliot-Th\'el\`ene
 (see  \cite[Thm.~3.5.7]{Serre}) $G/H$ is of Hilbert type.

\begin{notation}
Let $k$ be a field.
By $\kbar$ we denote a fixed algebraic closure of $k$.
If $X$ is a $k$-variety, we write $\Xbar:=X\times_k \kbar$.
If $H$ is a linear algebraic group (not necessarily connected or smooth),
we write $H^\mult$ for the largest quotient group of $H$ which is a group of multiplicative type.

Let $G$ be a smooth connected linear algebraic $k$-group, and $H\subset G$ a $k$-subgroup, not necessarily connected or smooth.
We write $G/H$ for the quotient variety. For the existence  of $G/H$
see \cite[Thm.~12.2.1]{Springer} in the case when $H$ is smooth,
and \cite[Exp.~VI$_A$, Thm.~3.2(i)]{SGA3} in the general case.
Since $G$ is smooth, so is $G/H$, see \cite[after Example A.1.12, bottom of p.~395]{CGP}.

If $k$ is a number field, we denote by $\sV(k)$ the set of all places of $k$, and by $\Om$ the set of all archimedean places.
If $v\in \sV(k)$, we write $k_v$ for the completion of $k$ at $v$.
We set $\kinf:=\prod_{v\in\Om} k_v$.
If $X$ is a $k$-variety, then $X(k)$ embeds into $X(\kinf)=\prod_{v\in\Om} X(k_v)$.
\end{notation}

\section{Number fields}\label{s:number}

\begin{notation}
Let $k$ be a field of characteristic 0.
Let $H$ be a linear $k$-group.
We assume that the kernel $H_1:=\ker[H\to H^\mult]$ is  connected and geometrically character-free.
In this case $H_1$ is an extension of a connected  semisimple group by a unipotent group.
Note that  in characteristic 0, any {\em connected} linear $k$-group $H$ and any {\em abelian} linear $k$-group $H$ satisfy this assumption.

Let $G$ be a connected linear algebraic $k$-group.
We use the following notation:

$G^\uu$ is the unipotent radical of $G$;

$G^\red=G/G^\uu$, it is reductive;

$G^\sss=[G^\red, G^\red]$, it is semisimple;

$G^\tor=G^\red/G^\sss$, it is a torus;

$G^\ssu=\ker[G\to G^\tor]$, it is an extension of the semisimple group $G^\sss$ by the unipotent group $G^\uu$.

Note that $G^\tor$ is the largest quotient torus of $G$ and that $G^\mult=G^\tor$.

\end{notation}

\begin{theorem}\label{t:quotient}
Let $k$ be a number field.
Let $G$ be a connected linear algebraic $k$-group, and let $H\subset G$ be a $k$-subgroup, not necessarily connected.
We assume that $G^\sss$ is simply connected and that $H_1:=\ker[H\to H^\mult]$
is  connected and geometrically character-free.
Then the variety $X:=G/H$ is of Hilbert type.
\end{theorem}

We need two lemmas.

\begin{lemma}\label{l:Arno}
Let $F$ be a linear algebraic group over a field $k$, and let $\Lambda$ be a subgroup of finite index of $F(k)$.
If $F$ is of Hilbert type, then the set $\Lambda\subset F(k)$ is not thin.
\end{lemma}

\begin{proof}
Since $\Lambda$ is of finite index in $F(k)$,
we have $F(k)=f_1\Lambda\cup\dots \cup f_m\Lambda$
for some finite set $f_1,\dots,f_m\subset F(k)$.
Suppose for the sake of contradiction that $\Lambda$ is thin,
then all the sets $f_i\Lambda$ are thin $(i=1,\dots, m)$, hence the finite union $F(k)=f_1\Lambda\cup\dots \cup f_m\Lambda$ is thin,
which contradicts to the assumption  that $F$ is of Hilbert type.
\end{proof}

\begin{lemma}\label{l:non-thin}
Let $T$ be a $k$-torus over a number field $k$.
Let $T(\kinf)^0$ denote the identity component of $T(\kinf)$.
Then the set $T(k)\cap T(\kinf)^0$ is not thin.
\end{lemma}

\begin{proof}
Since  the number field $k$ is Hilbertian, the torus $T$ is a variety of Hilbert type, cf. \cite[Cor.~7.14]{CS}.
Consider the group of connected components $\pi_0(T(\kinf)):=T(\kinf)/T(\kinf)^0$, it is finite.
It follows that the subgroup $T(k)\cap T(\kinf)^0$ is of finite index in $T(k)$.
Since $T$ is of Hilbert type, by Lemma \ref{l:Arno}  $T(k)\cap T(\kinf)^0$ is not thin.
\end{proof}

\begin{proof}[Proof of Theorem \ref{t:quotient}]
We prove the theorem in two steps, using the method of \cite{Borovoi-Crelle}.
\ms
{\em Step 1.}
The inclusion $H\into G$ induces a homomorphism $H^\mult\to G^\tor$, which need not be injective.
Choose an embedding $j\colon H^\mult\into Q$ into a quasi-trivial $k$-torus $Q$,
and denote by  $m\colon H\to H^\mult$  the canonical epimorphism.
Set $G_Y:=G\times_k Q$, then we have a diagonal embedding
$$
H\into G_Y,\quad h\mapsto (\,h,j(m(h))\,)\in G\times_k Q=G_Y.
$$
Let $H_Y$ denote the image of $H$ in $G_Y$, then $(H_Y)^\mult$ embeds into $(G_Y)^\tor$.
Set $Y=G_Y/H_Y$.
We have a map
$$
\pi\colon Y\to X,\ (g,q)H_Y\mapsto gH, \text{where } (g,q)\in G\times_k Q=G_Y.
$$
The subgroup $Q\subset G_Y$ acts on $Y$ on the left, and $Y$ is a left torsor over $X$ under $Q$.
It follows that all the fibers of $\pi$ are geometrically integral, in particular, the generic fiber of $\pi$ is geometrically integral.
Therefore, by \cite[Prop.~7.13]{CS}, in order to prove that $X$ is of Hilbert type, it suffices to show that $Y$ is of Hilbert type.
\medskip

{\em Step 2.}
From now on we assume that $H^\mult$ embeds into $G^\tor$, and we regard $H^\mult$ as a subgroup of $G^\tor$.
The canonical homomorphism $G\to G^\tor$ takes $H$ to $H^\mult\subset G^\tor$
and induces a $G$-equivariant map $\mu\colon X=G/H\to T:=G^\tor/H^\mult$.
We see that $T$ is a $k$-torus.
We have a commutative diagram
\[
\xymatrix{
G\ar[r]\ar[d]_t\ar[rd]^\phi    &X\ar[d]^\mu\\
G^\tor \ar[r]                 &T
 }
 \]
where the horizontal arrows are the canonical maps $G\to X:=G/H$ and $G^\tor\to T:=G^\tor/H^\mult$.
 Note that the diagonal arrow $\phi\colon G\to T$ is a homomorphism of $k$-groups.
We see from the diagram that $\mu(X(\kinf))\supset \phi(G(\kinf))$.
On the other hand, the homomorphism $\phi$ is surjective, hence it is a smooth morphism of varieties,
and therefore the subgroup $\phi(G(\kinf))\subset T(\kinf)$ is open,
hence it contains $T(\kinf)^0$.
Thus $\mu(X(\kinf))$ contains $T(\kinf)^0$.

 Let $t\in T(k)\cap T(\kinf)^0$ be a $k$-point.
We show that the fiber $X_t:=\mu^{-1}(t)$ is of Hilbert type.
First we prove that $X_t$ has a $k$-point.
Clearly $X_t$ is a homogeneous space of $G^\ssu:=\ker[G\to G^\tor]$.
From the assumption that $G^\sss$ is simply connected we deduce that $G^\ssu$ is simply connected.
Let $x_0\in X(k)$ denote the image of the unit element of $G$ in $X(k)=G(k)/H(k)$,
then the stabilizer of $x_0$ in $G^\ssu$ is $H\cap G^\ssu=H_1$, because $H^\mult$ embeds into $G^\tor$.
Let $\xbar\in X_t(\kbar)$
and let $\Hbar_{1,\xbar}$ denote the stabilizer of $\xbar$ in $\overline{G^\ssu}$, then
 $\Hbar_{1,\xbar}$ is conjugate to $\overline{H_1}$ in $\Gbar$, hence it is connected and character-free.
Since $t\in T(k)\cap T(\kinf)^0\subset\mu(X(\kinf))$, we see that the homogeneous space $X_t$ has a $k_v$-point for every $v\in \Om$.
By a local-global principle for such homogeneous spaces, see  \cite[Prop.~3.4(i)]{Borovoi-Crelle},
the homogeneous space $X_t$ has a $k$-point $x$.
Now let $H_{1,x}$ denote the stabilizer of $x$ in $G^\ssu$, then $H_{1,x}$ is connected.
Since $X_t$ is isomorphic to $G^\ssu/H_{1,x}$, where  $H_{1,x}$ is a connected $k$-subgroup of a connected  linear $k$-group $G^\ssu$,
by Proposition \ref{p:quot-connected} $X_t$ is of Hilbert type.

By Lemma \ref{l:non-thin} the subset $T(k)\cap T(\kinf)^0\subset T(k)$ is not thin, and we have shown that
for any $t\in T(k)\cap T(\kinf)^0$ the fiber $X_t$  of $X$ over $t$ is of Hilbert type.
By \cite[Theorem 1.1]{BFP}, the variety $X$ is of Hilbert type.
\end{proof}

\section{Function fields}\label{s:function}

Let $k$ be a field of characteristic $p>0$.
Let $G$ be a  smooth connected {\em reductive} linear algebraic group over $k$.
We set $G^\sss=[G, G]$ and $G^\tor=G/G^\sss$.

Recall that a global function field is the function field of a smooth, projective, connected curve over a finite field.

\begin{theorem}\label{t:quotient-pos}
Let $k$ be a global function field.
Let $G$ be a smooth, connected, reductive $k$-group, and let $H\subset G$ be a  $k$-subgroup, not necessarily connected or smooth.
We assume that $G^\sss$ is simply connected and that $H_1:=\ker[H\to H^\mult]$
is smooth, connected and semisimple.
Then the quotient variety $X:=G/H$  is of Hilbert type.
\end{theorem}

\begin{proof}
\ \ {\em Step 1.}
As in the proof of Theorem \ref{t:quotient},  we reduce the theorem to the case when $H^\mult$ embeds into $G^\tor$.
\medskip

{\em Step 2.}
From now on we assume that $H^\mult$ embeds into $G^\tor$, and we regard $H^\mult$ as a subgroup of $G^\tor$.
The canonical homomorphism $G\to G^\tor$ takes $H$ to $H^\mult\subset G^\tor$
and induces a $G$-equivariant map $\mu\colon X\to T:=G^\tor/H^\mult$.
We see that $T$ is a $k$-torus, hence $T$ is of Hilbert type, cf. \cite[Cor.~7.14]{CS}.

Let $t\in T(k)$ be a $k$-point.
We show that the fiber $X_t:=\mu^{-1}(t)$ is of Hilbert type.
Clearly $X_t$ is a homogeneous space of $G^\sss:=\ker[G\to G^\tor]$.
Let $\xbar\in X_t(\kbar)$.
Since  $H^\mult$ embeds into $G^\tor$, the stabilizer $\Hbar_{1,\xbar}$ of $\xbar$ in $\overline{G^\sss}$
is conjugate in $\Gbar$ to $\Hbar\cap\overline{G^\sss}=\overline{H_1}$, hence it is smooth, connected and semisimple.
Now by Proposition \ref{t:B-Douai} below,
the homogeneous space $X_t$ has a $k$-point $x$.
Let $H_{1,x}$ denote the stabilizer of $x$ in $G^\sss$, then $H_{1,x}$ is smooth and connected.
Since $X_t$ is isomorphic to $G^\sss/H_{1,x}$ where  $H_{1,x}$ is a smooth connected $k$-subgroup
of a smooth connected reductive $k$-group $G^\sss$,
by Proposition \ref{p:quot-connected} $X_t$ is of Hilbert type.

Since  $T$ is of Hilbert type and for any $t\in T(k)$ the fiber $X_t$ of $X$ over $t$ is of Hilbert type,
by \cite[Cor. 1.2]{BFP} the variety $X$ is of Hilbert type.
\end{proof}

In the proof of Theorem \ref{t:quotient-pos} we used the following  result for function fields, which is
similar to \cite[Prop.~3.4(i)]{Borovoi-Crelle} for number fields.

\begin{proposition}\label{t:B-Douai}
Let $k$ be a global function field. Let $G$  be a smooth, connected,
semisimple, simply connected $k$-group. Let $X$ be a right homogeneous
space of $G$. Let $\Hbar$ be the stabilizer of a geometric point $\xbar\in X(\kbar)$,
where $\kbar$ is a fixed algebraic closure of $k$.
Assume that $\Hbar$ is smooth, connected and semisimple.
Then $X$ has a $k$-point.
\end{proposition}

\begin{proof}
For a  commutative \'etale algebra $L/K$, we consider the category of pairs $(P_{(L)},\alpha_{(L)})$,
 where $P_{(L)}$ is a   right torsor over $L$ of $G_L:=G\times_k L$ and
$\alpha_{(L)}\colon P_{(L)}\to X_L$ is a $G_L$-equivariant $L$-morphism.
This category is clearly a groupoid (maybe empty).
Over the algebraic closure $\kbar$ of $k$ such a pair  clearly exists and any two such pairs are isomorphic.
We obtain a fibered category $\sG_X$ over $\Spec k$, which is
a gerbe  locally bound by $\Hbar$, see Giraud \cite[Ch.~IV, 5.1.2]{Giraud}.
Since $\Hbar$ is semisimple, the band of $\sG$
can be represented by a quasi-split $k$-form $H_0$ of $\Hbar$, see \cite[Prop.~V.3.2]{Douai-thesis},
and the gerbe $\sG$ is bound by $H_0$.
By definition, the gerbe  $\sG_X$ is neutral (trivial) if a pair $(P,\alpha)$ as above exists over $k$,
see Giraud \cite[Ch.~IV, Prop.~5.1.4(ii)]{Giraud}.
For details see also \cite[\S\,1]{Douai95}.
Since $k$ is a global function field, and since $H_0$ is smooth, connected and semisimple,
by Douai's theorem  \cite[Cor.~VIII.1.4]{Douai-thesis}
any gerbe  over $k$ which is bound by $H_0$ is neutral.
Thus our gerbe $\sG_X$ is neutral, i.e., there exists  pair $(P,\alpha)$ as above, defined over $k$.
By Harder's theorem \cite[Satz A]{Harder} any torsor of a simply connected semisimple $k$-group
over a global function field $k$ has a $k$-point.
Thus $P$ has a $k$-point $p$. We obtain a $k$-point $x$ of $X$ by taking  $x=\alpha(p)\in X(k)$.
This completes the proofs of Proposition \ref{t:B-Douai} and Theorem \ref{t:quotient-pos}.
\end{proof}

\section{Weak weak approximation}\label{s:weak}
In this section we give an alternative proof of Theorem \ref{t:quotient}, using the WWA (weak weak approximation) property.
In this section $k$ is always a number field and $S\subset \sV(k)$ is a finite subset.

One says that a $k$-variety $X$ has the WWA (weak weak approximation) property,
if there exists a finite subset $S_0\subset \sV(k)$ with the following property:
for any finite subset $S\subset \sV(k)$ such that $S\cap S_0=\emptyset$,
the variety $X$ has weak approximation in $S$
(i.e. $X(k)$ is dense in $\prod_{v\in S} X(k_v)$).
By a theorem of T.~Ekedahl and J.-L.~Colliot-Th\'el\`ene, see  \cite[Thm.~3.5.7]{Serre},
any variety with the WWA property is of Hilbert type.
We shall prove that $G/H$ as in Theorem \ref{t:quotient} has $WWA$, thus it is of Hilbert type.

\begin{notation}
Let $X$ be a smooth $k$-variety.
We write $\Br(X)$ for the cohomological Brauer group of $X$,
i.e. $\Br(X)=H^2_{\text{\'et}}(X,\GG_m)$.
We set $\Bro(X)=\ker[\Br(X)\to\Br(\Xbar)]$ where $\Xbar=X\times_k \kbar$.
We define the \emph{algebraic Brauer group} $\Bra(X)$
by $\Bra(X)=\coker[\Br(k)\to\Bro(X)]$.
For a finite subset $S\subset\sV(k)$
we set
$$
\Be_S(X)=\ker\left[\Bra(X)\to\prod_{v\notin S} \Bra(X_{k_v})\right].
$$
Set   $\Be_{S,\emptyset}(X)=\Be_S(X)/\Be_\emptyset(X)$.

Let  $B$ be a discrete $\Gal(\kbar/k)$-module
which is finitely generated as an abelian group (we say just ``a finitely generated Galois module'').
We write
$$
\Sh^i_S(k,B)=\ker\left[ H^i(k, B)\to \prod_{v\notin S} H^i(k_v,B)\right],
$$
where $i=1,2,\dots$.
We set $\Sh^i_{S,\emptyset}(k,B):=\Sh^i_S(k,B)/\Sh^i_\emptyset(k,B)$.

Let $A\to B$ be a morphism of finitely generated $\Gal(\kbar/k)$-modules.
We write $\H^i(k,A\to B)$  $(i=1,2,\dots)$ for  $i$-th Galois hypercohomology of the complex $A\to B$,
where $A$ is in degree 0 and $B$ is in degree 1.
We define  $\Sh^i_S(k,A\to B)$ and   $\Sh^i_{S,\emptyset}(k,A\to B)$ as above.
\end{notation}

\begin{lemma}\label{l:Sha}
Let $k$ be a number field, and let $A\to B$ be a homomorphism of finitely generated $\Gal(\kbar/k)$-modules,
where $A$ is free as an abelian group.
Let $K/k$ be the finite Galois extension in $\kbar$ corresponding to $\ker\left[\Gal(\kbar/k)\to \Aut A\times\Aut B\right]$.
Let $S\subset\sV(k)$ be a finite subset formed by places with cyclic decomposition groups in $\Gal(K/k)$.
Then $\Sh^2_{S,\emptyset}(k, A\to B)=0$.
\end{lemma}

\begin{proof}
Set $\Gamma=\Gal(K/k)$.
Choose an epimorphism of $\Gamma$-modules $P\onto B$, where $P$ is a permutation $\Gamma$-module.
Set $L=A\times_B P$, then we have a quasi-isomorphism $(L\to P)\to (A\to B)$,
and therefore isomorphisms
$$\H^2(k,L\to P)\isoto \H^2(k, A\to B) \quad \text{and} \quad  \Sh^2_{S,\emptyset}(k,L\to P)\isoto \Sh^2_{S,\emptyset}(k, A\to B).$$

We have a short exact sequence of complexes
$$
0\to (0\to P)\to (L\to P)\to (L\to 0)\to 0,
$$
which induces a hypercohomology exact sequence
$$
0=H^1(k,P)\to \H^2(k,L\to P)\to H^2(k,L)\to H^2(k,P),
$$
which in turn induces isomorphisms
$$\Sh^2_{S}(k,L\to P)\isoto \Sh^2_{S}(k, L)\quad \text{and}\quad\Sh^2_{S,\emptyset}(k,L\to P)\isoto \Sh^2_{S,\emptyset}(k, L)$$
(because $\Sh^2_S(k,P)=0$ by \cite[(1.9.1)]{Sansuc}).

Since $A$ and $P$ are finitely generated free abelian groups,
$L$ is a finitely generated free abelian group as well.
 We set $L^\vee:=\Hom(L,\Z)$.
 Choose an epimorphism $P'\onto L^\vee$, where $P'$ is a permutation $\Gamma$-module.
We obtain an embedding $L\into P''$, where $P'':=(P')^\vee$ is a permutation $\Gamma$-module.
Set $C=P''/L$.
The short exact sequence of $\Gamma$-modules
$$
0\to L\to P''\to C\to 0
$$
induces a cohomology exact sequence
$$
0=H^1(k,P'')\to H^1(k,C)\to H^2(k,L)\to H^2(k,P''),
$$
which in turn induces isomorphisms
$$\Sh^1_{S}(k,C)\isoto \Sh^2_{S}(k, L)\quad \text{and}\quad\Sh^1_{S,\emptyset}(k,C)\isoto \Sh^2_{S,\emptyset}(k, L)$$
(because $\Sh^2_S(k,P'')=0$ by \cite[(1.9.1)]{Sansuc}).

Since $\Gal(\kbar/K)$ acts trivially on $C$, and $S$ consists of  places with cyclic decomposition groups in $\Gal(K/k)$,
  by \cite[Cor.~3.3]{Borovoi-Ramanujan} $\Sh^1_{S,\emptyset}(k,C)=0$.
Thus $\Sh^2_{S,\emptyset}(k, A\to B)=0$, which proves the lemma.
\end{proof}

\begin{proof}[Alternative proof of Theorem \ref{t:quotient}]
Let $G$ and $H$ be as in Theorem \ref{t:quotient}.
We have a homomorphism of $\Gal(\kbar/k)$-modules $\Ghat\to \Hhat$,
where $\Ghat:=\Hom(\Gbar,\GG_{m,\kbar})$ and $\Hhat:=\Hom(\Hbar,\GG_{m,\kbar})$ are the corresponding geometric character groups.
Let $K/k$ be the finite  Galois extension in $\kbar$ corresponding to the kernel
$$
\ker\left[\Gal(\kbar/k)\to\Aut \Ghat \times \Aut\Hhat\right].
$$
Let $S_0\subset\sV(k)$ denote the (finite) set of all places of $k$ with {\em non-cyclic}
decomposition groups in $\Gal(K/k)$.
Let $S\subset\sV(k)$ be a finite set such that $S\cap S_0=\emptyset$.
We shall prove that $G/H$ has weak approximation in $S$.

By \cite[Thm.~7.2]{BvH2} there is a canonical isomorphism $\Bra(G/H)\isoto\H^2(k,\Ghat\to\Hhat)$,
which induces a canonical isomorphism
$\Be_{S,\emptyset}(G/H)\isoto\Sh^2_{S,\emptyset}(k,\Ghat\to\Hhat)$.
By Lemma \ref{l:Sha} $\Sh^2_{S,\emptyset}(k,\Ghat\to\Hhat)=0$, hence $\Be_{S,\emptyset}(G/H)=0$.

The Brauer-Manin obstruction of \cite{Borovoi-Crelle} to weak approximation in $S$ for a $k$-variety $X$
is a certain map
$$
m_S\colon \prod_{v\in S}X(k_v)\to \Be_{S,\emptyset}(G/H)^D,
$$
where $\Be_{S,\emptyset}(G/H)^D:=\Hom(\Be_{S,\emptyset}(G/H),\Q/\Z)$.
Since $\Be_{S,\emptyset}(G/H)=0$, we obtain that $\Be_{S,\emptyset}(G/H)^D=0$, hence $m_S$ is identically zero.
We see that the group $G^\sss$ is simply connected, $H_1:=\ker[H\to H^\mult]$ is connected and geometrically character-free,
and the obstruction $m_S$ is identically zero, hence
by \cite[Thm.~2.3]{Borovoi-Crelle}  the variety $G/H$ has  weak approximation in $S$.

We have proved that $G/H$ has WWA, hence it is of Hilbert type.
\end{proof}

\noindent
{\bf Acknowledgements.} The author is grateful
to Lior Bary-Soroker, Jean-Claude Douai, Arno Fehm and Cristian D. Gonz\'alez-Avil\'es  for very helpful discussions.

\end{document}